\documentclass[12pt]{amsart}
\usepackage{amstext,amsfonts,amssymb,amscd,amsbsy,amsmath,verbatim}
\usepackage{ifthen}
\usepackage{color,tikz}
\usepackage{amsthm}
\usepackage{latexsym}
\usepackage[all]{xy}
\usepackage{enumerate}
\usepackage{url}

\newtheorem{lemma}{Lemma}[section]
\newtheorem{theorem}[lemma]{Theorem}

\newtheorem{prop}[lemma]{Proposition}

\newtheorem{claim*}{Claim}
\newtheorem{thm}[lemma]{Theorem}

\newtheorem{example}[lemma]{Example}

\theoremstyle{remark}
\newtheorem{remark}[lemma]{Remark}

\usepackage{geometry,enumerate}
\newcommand{\PP}{\mathbb{P}}

\newcommand{\GL}{{GL}}

\title[Blow-ups of $\mathbb{P}^{n-3}$ at $n$ points]{Blow-ups of $\mathbb{P}^{n-3}$ at $n$~points
 and spinor varieties}
\author{Bernd Sturmfels}
\author{Mauricio Velasco}
\address{Department of Mathematics, University of California,
	Berkeley, CA 94720, USA}
\email{bernd@math.berkeley.edu, velasco@math.berkeley.edu}
\urladdr{\ \ math.berkeley.edu/\~{}bernd, \ math.berkeley.edu/\~{}velasco}
\thanks{Bernd Sturmfels is partially supported by NSF grants DMS-0456960 and DMS-0757236. \\
Mauricio Velasco is partially supported by NSF grant DMS-0802851.}

\begin{document}

\begin{abstract}
Work of Dolgachev and  Castravet-Tevelev establishes a bijection between the 
$2^{n-1}$ weights of the half-spin representations of $\mathfrak{so}_{2n}$ and the generators of the Cox ring of the 
variety $X_n$ which is obtained by blowing up $\mathbb{P}^{n-3}$ at $n$ points.
We derive a geometric explanation for this bijection, by embedding
 ${\rm Cox}(X_n)$ into the even spinor variety (the homogeneous space of the even
half-spin representation). The Cox ring of the blow-up $X_n$ is recovered
geometrically by intersecting torus translates of the even spinor variety.
These are higher-dimensional generalizations of results by Derenthal and Serganova-Skorobogatov on del Pezzo surfaces.
\end{abstract}

\maketitle

\section{Introduction}
In the early '90s Batyrev observed that the well known equality between the number of exceptional curves on Del Pezzo surfaces of degree $2\leq \delta\leq 5$ and the dimension of certain minuscule representations of the semisimple groups of type $A_4,D_5,E_6$ and $E_7$ has a geometric
explanation. He conjectured that the universal torsor over any Del Pezzo surface admits an embedding into the homogeneous space defined by the orbit of the highest weight vector of the representation. Batyrev's conjecture was proved independently by Derenthal~\cite{Derenthal-Advances} and by Serganova and Skorobogatov~\cite{SS1}. 

For del Pezzo surfaces of degree five, the universal torsor and the corresponding homogeneous space (the Grassmannian ${\rm Gr}(2,5)$) coincide. This coincidence suggests that it should be possible
to recover the universal torsor from the corresponding homogeneous space. However, there is an obvious difficulty: Del Pezzo surfaces of degree $\delta$ form a family of dimension $10-2\delta$, while the homogeneous space is unique. A key insight of Serganova and Skorobogatov~\cite{SS2} is that 
the universal torsor is recovered by intersecting several torus translates of the corresponding homogeneous space. The chosen elements in the torus are determined by the moduli of the surface. 

In this paper we extend these constructions from del Pezzo surfaces to the higher-dimensional varieties $X_n$ obtained by blowing up $\mathbb{P}^{n-3}$ at $n\geq 5$ general points. Work of 
Dolgachev~\cite{DO} and  Castravet-Tevelev \cite{CT} 
ensures that there is bijection between the
 $2^{n-1}$ generators of the Cox ring of $X_n$ 
and the $2^{n-1}$ weights of the half-spin representations of $\mathfrak{so}_{2n}$.
We here offer a geometric explanation for this bijection:

\begin{theorem} 
\label{thm:main} The spectrum of the Cox ring of $X_{n}$ can be embedded into the spinor variety 
$S^{+}$ in $\bigwedge^{even}W$, where $W \simeq k^n$. 
If $I_X$ denotes the homogeneous prime ideal
in the polynomial ring $k\left[ \bigwedge^{even}W\right]$ 
that presents Cox ring of $X_n$, then we have
\begin{equation}
\label{inclusion}
I_X\,\,\supseteq \,\,\sum_{c\in \mathcal{G}(p)} a(c)\star I_{{\rm spin}}.
\end{equation}
\end{theorem}

Here $I_{\rm spin}$ is the ideal defining the spinor variety $S^+$, the vector 
$a(c)$ has $2^{n-1}$ nonzero components which are explicit rational functions on 
a certain moduli space of point configurations, 
and $a(c) \star I_{\rm spin}$ denotes the ideal obtained from $I_{\rm spin}$ by scaling each variable 
in  $k\left[ \bigwedge^{even}W\right]$ by the corresponding entry in $a(c)$. We refer to Section~\ref{sec: final} for precise definitions.
We conjecture that equality holds in (\ref{inclusion}) for generic $X_n$, 
and that only two summands will suffice on the right hand side. 
This conjecture has been verified for $n \leq 8$ using computational
algebra methods (see Theorem~\ref{thm: generic}). 

All the ideals in (\ref{inclusion}) are generated by quadrics. Quadratic generation of the 
Cox ideal $I_X$ 
follows from the sagbi degenerations of Sturmfels-Xu~\cite{SX}, which relate the Cox rings of $X_n$ 
to the toric varieties studied by Buczy\'nska and Wi\'sniewski~\cite{BW}. These toric
degenerations represent statistical models for phylogenetic trees.
The spinor ideal $I_{\rm spin}$ is the prime ideal of all algebraic relations
among the $2^{n-1}$ subpfaffians of a skew-symmetric $n \times n$-matrix.
The quadratic generation of $I_{\rm spin}$ is a classical result from the literature
on algebras with straightening laws (cf.~De Concini-Procesi \cite{DP}), and we shall 
present the corresponding quadratic Gr\"obner basis in Section 6.

A main new idea in this paper is the construction
(in Section~\ref{sec: pfaffians}) of skew-symmetric
matrices whose subpfaffians generate the Cox ring of $X_n$.
These matrices enable us to extend the representation-theoretic approach of Serganova and Skorobogatov \cite{SS1, SS2} from del Pezzo surfaces to higher dimensions.
An important element in the proof
of Theorem \ref{thm:main} is a remarkable identity among pfaffians and determinants discovered by Okada~\cite{Okada} in connection with rectangular representations of the general linear group.
We shall review Okada's identity in Section~\ref{sec: identity}.
This furnishes the link between our Pfaffian generators for ${\rm Cox}(X_n)$
and  the determinantal generators given in \cite{CT}.

In Section \ref{sec:geometry} we start out with basic facts about the geometry of the 
blow-up varieties $X_n$ and their Cox rings, and we fix the notation and conventions used 
throughout this paper. In Section \ref{sec:trees} we present a result in combinatorial
commutative algebra that may be of independent interest: each phylogenetic tree 
specifies a degeneration of the Cox ring of $X_n$ to an algebra generated by Pl\"ucker monomials.
This refines the results on sagbi bases in \cite[\S 7]{SX}, and it opens up the possibility of 
relating our Cox ring to the subalgebras studied by Howard \textit{et al.}~\cite{HMSV} and
Manon~\cite{Man}. An important player in this connection should be the moduli space of rank two stable quasiparabolic bundles on $\PP^1$ with $n$ points (cf.~\cite[Theorem 7.2]{SX}).

 Another promising direction of inquiry would be to clarify
 the relationship between the remaining spaces $X_{a,b,c}$ studied by Castravet and Tevelev in~\cite{CT} and the homogeneous spaces of the fundamental representations of semisimple groups of type $T_{a,b,c}$.

\medskip

\noindent
{\bf Acknowledgements}. We thank Ana Maria Castravet, David Eisenbud, Vera Serganova, Damiano Testa and Anthony V\'arilly-Alvarado for helpful conversations.

\section{Geometry of blow-ups of $\mathbb{P}^{n-3}$ at $n$ points}\label{sec:geometry}

In this section we collect some facts about the geometry
and representation theory  relevant for
blow-ups of $\mathbb{P}^{n-3}$ at $n$ points and 
  their Cox rings. We also establish notation which will be used throughout the paper.
Let $k$ be an algebraically closed field.
For $n \geq 5$, let $X_{n}(Q)$ be the variety obtained by blowing up $\mathbb{P}^{n-3}$ at $n$ general points  $Q_1,\dots, Q_n$ and let $\pi: X_{n}(Q)\rightarrow \mathbb{P}^{n-3}$ be the canonical projection. 
The variety $X_n(Q)$ depends on the points $Q_1,\dots, Q_n$ up to projective equivalence. It follows that the moduli space of the varieties $X_n(Q)$ has dimension $n-3$. 

Let $\ell\subset \mathbb{P}^{n-3}$ be a hyperplane. The Picard classes $H,E_1,\dots, E_n$ with $H:=[\pi^{*}(\ell)]$ and $E_i=[\pi^{-1}(Q_i)]$ are a basis for ${\rm Pic}(X_{n}(Q))\cong \mathbb{Z}^{n+1}$. The canonical class is $\,K:=-(n-2)H+(n-4)\left(E_1+\dots +E_n\right)$. We endow ${\rm Pic}(X_n(Q))$ with a symmetric bilinear form via $H^2=n-4$, $E_iE_j=-\delta_{ij}$ and $HE_j=0$.
The set of classes in ${\rm Pic}(X_n(Q))$ which are orthogonal to $K$ and have square $-2$ form a root system of type $D_n$. A set of simple roots for this root system is given by
 $\{\alpha_1,\ldots, \alpha_n\}$, where
\[
\alpha_i \,\,\, =  \,\,\, \begin{cases}
\qquad E_i-E_{i+1} & \text{ if $1\leq i\leq n-1$,}\\
H-E_1-\dots -E_{n-2} & \text{ if $i=n$.}\\
\end{cases}
\]
The action of the Weyl group of this root system on the orthogonal complement of
the canonical divisor $K$ extends to an action on ${\rm Pic}(X_n(Q)) $ by fixing $K$. 

As shown in~\cite[Theorem 2]{DO}, the orbit of $E_n$ under the action of the Weyl group consists of $2^{n-1}$ classes of effective divisors which are exceptional on some small modification of $X_n(Q)$. The elements of this orbit are the {\em $(-1)$-divisors} on $X_n(Q)$.
The exceptional divisor $E_n$ is dual to the root $\alpha_{n-1}$ in the sense that $E_n\alpha_{j}=\delta_{j,n-1}$. As such, up to addition of a multiple of $K$, it coincides with the highest weight of a fundamental representation of the even orthogonal Lie algebra $\mathfrak{so}_{2n}$. As a consequence, the action of the Weyl group determines a bijection between the $(-1)$-divisors and the elements of the orbit of the highest weight of this representation. 

We now describe this bijection more explicitly. Fix a $2n$-dimensional vector space
$V$ with coordinates $x_1,\dots, x_n, y_1,\dots, y_n$.
Recall that $\mathfrak{so}_{2n}$ consists of the endomorphisms $A$
of $V$ which respect the quadratic form $Q(x,y)=\sum_{i=1}^n x_iy_i$, meaning that
 $Q(Av,w)+Q(v,Aw)\equiv 0$. Let $\mathfrak{h}\subset \mathfrak{so}_{2n}$ denote the subalgebra
 consisting of  all diagonal matrices $D={\rm diag}(d_1,\dots, d_n,-d_1,\dots, -d_n)$.
 We define a basis $L_1,\ldots,L_n$ of $\mathfrak{h}^*$ by the property
 $L_i(D)=d_i$, and we set $L_i\cdot L_j=-\delta_{ij}$,
 for all $1 \leq i,j \leq n$. 
 
 A system of simple roots of the Lie algebra $\mathfrak{so}_{2n}$ is given by
  $\{\beta_1,\ldots, \beta_n\}$, where
\[
\beta_i \,\,= \,\, \begin{cases}
\, L_{i}-L_{i+1} & \text{if $i \leq n-2$,} \\
\, L_{n-1}+L_n & \text{if $i=n-1$,}\\
\, L_{n-1}-L_n & \text{if $i=n$.}
\end{cases}
\]
The element $\omega_{n-1}=-\frac{1}{2}(\sum_{i=1}^{n}L_i)$ is dual to $\beta_{n-1}$.
This is the highest weight of the even spin representation $S^{+}$ (see~\cite[Lecture 20]{Fulton-Harris} for a construction of this representation). The underlying vector space of $S^{+}$ is $\bigwedge^{even}(W)$ where $W={\rm span}(f_1,\dots, f_n)$ and its weight vectors are the vectors $f_B=\wedge_{j\in B}f_j$ with weight $W(B):=\frac{1}{2}(\sum_{i\in B}L_i-\sum_{i\not \in B} L_i)$ parametrized by the even subsets $B\subseteq [n]$ (see~\cite[Lemma 20.15]{Fulton-Harris}). 

Next, we define a linear map $T: \mathfrak{h}^* \rightarrow K^{\perp}\subset {\rm Pic}(X_n(Q))$
by setting $T(\beta_i)=\alpha_i$. We use it to define a bijection between the 
weights of the even half-spin representation and the $(-1)$-divisors on $X_n$. Note that $T$ is an isometry so the identification below is compatible with the actions of the Weyl group of $D_n$ on $\mathfrak{h}^*$ and on ${\rm Pic}(X_n(Q))$.

\begin{lemma}\label{lem: bijection}  If $\,B\subset [n]\,$ with $\,|B|=2s\,$ then 
$\,T(W(B))=D(B)+\frac{1}{4}K$, where
\[D(B) \,\, = \,\, \, \begin{cases} \,\,\,\,
s\left(H-\sum_{i=1}^{n}E_i\right)\,+\,\sum_{b\in B\cup \{n\}}E_b & \text{if $n\not\in B$,}\\
(s-1)\left(H-\sum_{i=1}^{n}E_i\right)+\sum_{b\in B\setminus \{n\}}E_b & \text{if $n\in B$.}
\end{cases}
\]
\end{lemma}
\begin{proof} Let $\Delta:=H-\sum_{i=1}^nE_i$ and note that $T(L_n)=-E_n-\frac{1}{2}\Delta$ and that $T(L_i)=E_i+\frac{1}{2}\Delta$ for $1\leq i\leq n-1$. To show the above equality we study two cases depending on whether or not the set $ B$ contains the index $n$.
If $n\in B$ then $T(W(B))$ equals
\[\frac{1}{2}\! \left(\sum_{b\in B\setminus \{n\}} \!\! (E_b+\frac{1}{2}\Delta) - \!\!
\!\! \sum_{b\in B^{c}\cup \{n\}} \!\!\!\! (E_b+\frac{1}{2}\Delta) \right)= 
\frac{1}{2} \!\left(\frac{4s-n-2}{2}\Delta + \!\!\! \sum_{b\in B\setminus \{n\}}
\!\! E_b- \!\!\!\!\! \sum_{b\in B^{c}\cup \{n\}}\!\! E_b\right)\! .\]
Subtracting $D(B)$ from this expression, we obtain
\[\frac{1}{2}\left(\frac{(4s-n-2-4(s-1))}{2}\Delta \,\,\, -
\! \sum_{b\in B\setminus\{n\}}\!\! E_b \,- \! \sum_{b\in B^{c}\cup \{n\}} \!\! E_b \right)
\,\,= \,\, \frac{K}{4}.\]
Similarly, if $n\not\in B$ then $T(W(B))$ equals
\[\frac{1}{2} \! \left(\sum_{b\in B\cup \{n\}}\!\!\! (E_b+\frac{1}{2}\Delta)- \!\!\! \!
\sum_{b\in B^{c}\setminus \{n\}} \!\! (E_b+\frac{1}{2}\Delta) \right)
= \frac{1}{2} \! \left(\! \frac{4s-n+2}{2}\Delta +
\!\!\! \sum_{b\in B\cup \{n\}} \!\! E_b- \!\!\!\! \sum_{b\in B^{c}\setminus \{n\}} \!\!\! E_b\right)\!.\]
If we subtract $D(B)$ from this expression then we obtain $\frac{1}{4}K$ as claimed.
\end{proof}

We next review the definition of the Cox ring.
Let $X$ be any smooth projective variety with 
${\rm Pic}(X)\cong \mathbb{Z}^{n+1}$ and $D_0,\ldots, D_n$ a collection of divisors whose classes form a basis for ${\rm Pic}(X)$. Then the Cox ring of $X$ is the ${\rm Pic}(X)$-graded algebra 
\begin{equation}
\label{eq:cox}
{\rm Cox}(X) \,\,\,
=\bigoplus_{(m_0,\ldots, m_n)}{\rm H}^0\bigl(X,\mathcal{O}_X(m_0D_1+\cdots +m_nD_n)\bigr).
\end{equation}
For $X=X_n(Q)$ we fix divisors $h,e_1,\dots, e_n$ in the classes $H,E_1,\dots, E_n$. 
The Cox ring of $X_n(Q)$ is realized as the subalgebra of $k[x_1,\dots, x_{n-2}][t_0^{\pm},\dots, t_{n}^{\pm}]$ given by  \[\bigoplus_{(m_0,m_1,\dots, m_n)\in \mathbb{Z}^{n+1}} 
\!\! \Gamma(m_0 H+m_1E_1+\cdots+m_n E_n)\big)\cdot t_0^{m_0} t_1^{m_1} \cdots t_n^{m_n}\]
Here $x_1,\ldots,x_{n-2}$ are  coordinates on $\mathbb{P}^{n-3}$ and
$\Gamma(m_0 H-m_1E_1-\cdots -m_nE_n)$ is the vector space consisting of homogeneous polynomials of total degree $m_0$ in the $x_i$ that vanish with multiplicity at least $m_i$ at the point $Q_i$.
 For an $(n+1)$-tuple $t=(t_0,t_1,\dots, t_n)$ and $D=m_0H+m_1E_1+\cdots +m_nE_n$ we define $t^D:=t_0^{m_0}t_1^{m_1}\cdots t_n^{m_n}$.

Castravet and Tevelev~\cite{CT} showed that the Cox ring of $X_n(Q)$ is generated
 as $k$-algebra  by any $2^{n-1}$ nonzero global sections supported on the $(-1)$-divisors. Any choice of such sections determines a presentation of the cox ring as a quotient of a polynomial ring by an ideal of relations. As shown by Stillman, Testa and Velasco for Del Pezzo surfaces~\cite{STV}, and by Sturmfels and Xu~\cite{SX} in general, these ideals of relations admit quadratic Gr\"obner bases and in particular are generated by quadrics. 

\section{An identity involving pfaffians and determinants} \label{sec: identity}
In this section we present a combinatorial identity discovered by Okada~\cite{Okada} which plays a central role in our approach.
  We work in the polynomial ring over $k$ with variables $X_i,Y_i,P_i,x_i,y_i,p_i$ for $1\leq i\leq n$. For $i,j\in [n]$ let $p_{ij}:=P_ip_j-P_jp_i$ and define $x_{ij}$ and $y_{ij}$ similarly.  
For an even index set $B=\{b_1<b_2<\dots< b_{2s}\}\subset [n]$,
 let $V_B(x,X,p,P)$, or $V_B(x,p)$ for brevity, be the $2s\times 2s$ matrix whose $m$-th row is  
\[
\bigl(\,
x_{{b_m}} p_{b_m}^{s-1}, \,
x_{{b_m}} p_{b_m}^{s-2}P_{b_m}, \, \ldots ,\,
 x_{{b_m}}P_{b_m}^{s-1} \,,\,
 X_{{b_m}} p_{b_m}^{s-1}, 
 \,X_{{b_m}} p_{b_m}^{s-2}P_{b_m}, \,\ldots , \,X_{{b_m}}P_{b_m}^{s-1} \,\bigr),
\]
and let $\Psi_B(X,x,P,p)$ (or $\Psi_B(x,p)$) denote its determinant. Note that $\Psi_{\emptyset}(x,p)=1$.

Let $\mathcal{A}(x,y,p)$ denote the skew-symmetric $n\times n$ matrix with off-diagonal entries
\[
\mathcal{A}(x,y,p)_{ij} \,\,\, = \,\,\, \frac{x_{ij}y_{ij}}{p_{ij}}
\,\,\, = \,\,\,\, \begin{vmatrix}
X_i & X_j\\
x_i & x_j\\
\end{vmatrix} \cdot
\begin{vmatrix}
Y_i & Y_j\\
y_i & y_j \\
\end{vmatrix} \cdot
\begin{vmatrix}
P_i & P_j\\
p_i & p_j\\
\end{vmatrix}^{-1}
\]
We write $\mathcal{A}_B(x,y,p)$ for the submatrix of $\mathcal{A}(x,y,p)$ obtained by choosing the rows and columns indexed by the elements of $B$. The matrix $\mathcal{A}_B(x,y,p)$ is a $2s\times 2s$ skew-symmetric matrix, and we denote its pfaffian by $F_B(x,y,p)$. The following identity, due to Okada~\cite{Okada}, relates the pfaffians and the determinants defined above.

\begin{theorem} \label{thm: identity} For any even subset $B=\{b_1<\dots <b_{2s}\}$ of $[n]$,
we have
\begin{equation}
\label{OkadaTwo} F_B(x,y,p) \,\,\, = \,\,\, \frac{\Psi_B(x,p)\Psi_B(y,p)}{\prod_{i<j\in B}p_{ij}}.
\end{equation}
In particular, if $1$ denotes the vector $(1,\dots, 1)$ of length  $n$, then this specializes to
\begin{equation}
 \label{OkadaOne} {\rm pfaff}\left(\frac{(X_{b_i}-X_{b_j})(Y_{b_i}-Y_{b_j})}{(P_{b_i}-P_{b_j})}\right)
 \,\, = \,\, \frac{\Psi_B(X,1,P,1)\Psi_B(Y,1,P,1)}{{\prod_{b_i<b_j}(P_{b_i}-P_{b_j})}}.
 \end{equation}
\end{theorem} 

\begin{proof} Equation \eqref{OkadaOne} is the special case $n=m$ of Theorem 4.7 in~\cite{Okada}. 
We can derive \eqref{OkadaTwo} from \eqref{OkadaOne} by a homogenization argument as follows.
Let $T_B$ denote the skew-symmetric $2s \times 2s$ matrix whose off-diagonal entries are given by
\[(T_B)_{ij} \,\,\,= \,\, \left(\frac{X_{b_i}}{x_{b_i}}-\frac{X_{b_j}}{x_{b_j}}\right) \cdot
\left(\frac{Y_{b_i}}{y_{b_i}}-\frac{Y_{b_j}}{y_{b_j}}\right) \cdot
 \left(\frac{P_{b_i}}{p_{b_i}}-\frac{P_{b_j}}{p_{b_j}}\right)^{\! -1} .\]
Let $D$ be the diagonal matrix with $D_i= x_{b_i}y_{b_i} p_{b_i}^{-1}$. Then $\mathcal{A}_B(x,y,p)=D^tT_BD$, and the left hand side of~\eqref{OkadaTwo} equals 
$\,\det(D) \cdot {\rm pfaff}(T)$. Using the identity~\eqref{OkadaOne} we obtain
\[F_B(x,y,p) \,\, = \,\,\,\det(D) \cdot \frac{\det V_B(\frac{P}{p},1,\frac{X}{x},1 )
\cdot \det V_B(\frac{P}{p},1,\frac{Y}{y},1)}{\prod_{i<j}\left(\frac{P_{b_i}}{p_{b_i}}-\frac{P_{b_j}}{p_{b_j}}\right)} . \]
In the left factor of the numerator we now substitute the expression
\[
\det V_B\left(\frac{X}{x},1,\frac{P}{p},1\right) \,\,\,=
\,\,\, \Psi_B(x,p) \cdot \left(\prod_{j=1}^{2s}x_{b_j}p_{b_j}^{s-1}\right),
\] 
and similarly for the right factor.
After some cancellations, identity~\eqref{OkadaTwo} emerges.
\end{proof}

\begin{remark} An irreducible representation of $\GL(n,\mathbb{C})$ is {\em rectangular} if the corresponding partition has parts of equal size. Okada found the above identity
 in connection with tensor products of rectangular representations. When $n=2m$ and the rectangular partition has exactly $m$ parts of length $s$, the character of the representation is the 
 Schur function $\frac{\Psi(X,1,X^{s+m},1)}{\Delta}$ where $\Delta$ is a Vandermonde determinant. 
 The identity~\eqref{OkadaOne} expresses the character of the tensor product of two such rectangular representations as a Pfaffian. The minor summation formula then can be used to find
 the decomposition of this tensor product into irreducibles \cite[Theorem 2.4]{Okada}.
\end{remark}

\section{Pfaffian generators for the Cox ring}\label{sec: pfaffians}
The points $Q_1,\dots, Q_n$ are assumed to be in linearly general position in $\mathbb{P}^{n-3}$.
We can thus choose coordinates $\,x_1,x_2,\ldots, x_{n-2}\,$ so that $\,Q_1,\dots, Q_{n-2}\,$ are the canonical basis vectors, $\,Q_{n-1}=[1:1:\dots:1]$, and $\,Q_n=[p_1:p_2:\dots:p_{n-2}]$ 
for some $p_i\in k$. 

We now show that, in the chosen coordinates, the $2^{n-1}$ hypersurfaces  whose strict transforms yield the $(-1)$-divisors are defined by the subpfaffians of an $n\times n$ skew-symmetric matrix. These defining equations are unique only up to scalar multiplication. To specify the scalars we use an additional parameter $y\in \mathbb{A}^{n-2}_k$ not lying in any of the hypersurfaces. 
Let $M$ denote the skew-symmetric $n \times n$ matrix 
\[ M \,\, =  \,\, \begin{pmatrix}
\,0 & \frac{(x_2-x_1)(y_2-y_1)}{(p_2-p_1)} & \frac{(x_3-x_1)(y_3-y_1)}{(p_3-p_1)} & \cdots & \frac{(x_{n-2}-x_1)(y_{n-2}-y_1)}{(p_{n-2}-p_1)} & -\frac{x_1y_1}{p_1} &1\\
\,\vdots & 0 & \frac{(x_3-x_2)(y_3-y_2)}{(p_3-p_2)} & \cdots & \frac{(x_{n-2}-x_2)(y_{n-2}-y_2)}{(p_{n-2}-p_{2})} & -\frac{x_2y_2}{p_2} &1\\ \,
\vdots & \vdots & 0 & \cdots & \frac{(x_{n-2}-x_3)(y_{n-2}-y_3)}{(p_{n-2}-p_3)} & -\frac{x_3y_3}{p_3} &1\\
\,\vdots & \vdots & \vdots & \cdots & \vdots & \vdots\\
-1 & -1 & -1 &\cdots & -1 & -1 & 0\\
\end{pmatrix}.
\] 
For $B\subset [n]$ let $M_B$ be the square submatrix with 
rows and columns indexed by $B$.

\begin{lemma}\label{lem: hypersurfaces} Let $B\subset [n]$ be an even subset. The pfaffian ${\rm pfaff}(M_B)$ is a nonzero element of $\Gamma(D)$ where $D$ is the $(-1)$-divisor corresponding to $B$ as in Lemma~\ref{lem: bijection}.
\end{lemma}

\begin{proof} Suppose $|B|=2s$. For  $1\leq i<j\leq n$ we have 
\[M_{ij} \,\in \,
\begin{cases}
\Gamma\bigl((H-\sum_{l=1}^{n-1} E_l)+E_i+E_j \bigr) & \text{if $j\neq n$,} \\
\qquad \qquad \Gamma(E_i) & \text{if $j=n$.}
\end{cases}
\]
The pfaffian of the $2s \times 2s$-submatrix $M_B$ has the expansion
\[{\rm pfaff}(M_B) \,\,\, = \,\,\, \sum_{\mu} {\rm sign}(\sigma(\mu))\prod_{(a,b) \in \mu} \!\! M_{ab}\]
where $\mu$ runs over all matchings of $B$. All terms on the right hand side belong to 
$\Gamma \bigl((s-\delta_B)(H \! - \! \sum_{l=1}^{n-1} E_l) +\sum_{b\in B\setminus \{n\}} \! E_b \bigr)$ 
where $\delta_{B}=1$ if $n \in B$ and $\delta_{B}=0$  otherwise.

We now show that the pfaffian of $M_B$ vanishes to order at least $s-1$ at the point $Q_n$. When $s=0$ the statement is trivial. For $s>0$ we show that, for any $(v_1,\dots, v_{n-2})$, the pfaffian evaluated at $x_i=p_i+\epsilon v_i$ is divisible by $\epsilon^{s-1}$. Consider the $n {\times} n$ matrix
\[\alpha \,\,\, = \,\,\,
\begin{pmatrix}
0 & y_2-y_1 & \cdots & \cdots & y_{n-2}-y_1 & -y_1 & 1\\
y_1-y_2 & 0 & y_3-y_2 & \cdots & y_{n-2}-y_2 & -y_2 & 1\\
y_1-y_3 & y_2-y_3 & 0 & \cdots & y_{n-2} -y_3 & -y_3 & 1 \\
\vdots & \vdots & \vdots & \ddots & \vdots & \vdots & \vdots\\
y_1-y_{n-2} & y_2-y_{n-2} &  y_3-y_{n-2}   & \cdots &   0 & -y_{n-2} & 1\\
y_1 & y_2 & y_3 & \cdots & y_{n-2} & 0 & 1 \\
-1 & -1 & -1 & \cdots & -1 & -1 & 0\\
\end{pmatrix}
\]
We regard $\alpha$ as an exterior tensor of step $2$.
Then its $m$-th exterior power $\alpha^{(m)}$ is zero for $m \geq 2$ because
the coordinates $\alpha_{ij}$ of $\alpha$  are the $2\times 2$ minors of the matrix 
\[
\begin{pmatrix}
    1 & 1  & \cdots & 1 & 1 & 0\\
y_1 & y_2 & \cdots &  y_{n-2} & 0 & 1\\
\end{pmatrix}
.\]
We also consider the following skew-symmetric matrix an an exterior tensor of step~$2$:
\[
\beta \,\, = \,\, \begin{pmatrix}
\, 0 & \frac{(v_2-v_1)(y_2-y_1)}{(p_2-p_1)} & \frac{(v_3-v_1)(y_3-y_1)}{(p_3-p_1)} & \cdots & \frac{(v_{n-2}-v_1)(y_{n-2}-y_1)}{(p_{n-2}-p_1)} & -\frac{v_1y_1}{p_1} & 0 \,\\
\,\vdots & 0 & \frac{(v_3-v_2)(y_3-y_2)}{(p_3-p_2)} & \dots & \frac{(v_{n-2}-v_2)(y_{n-2}-y_2)}{(p_{n-2}-p_{2})} & -\frac{v_2y_2}{p_2} & 0 \, \\ \,
\vdots & \vdots & 0 & \ddots & \frac{(v_{n-2}-v_3)(y_{n-2}-y_3)}{(p_{n-2}-p_3)} & -\frac{v_3y_3}{p_3}
 & 0 \, \\ \vdots & \vdots & \vdots & \cdots & \vdots & \vdots & \vdots \,\\ \,
0 & 0 & 0 &\dots & 0 & 0 & 0 \,\\
\end{pmatrix}.
\] 
In this notation, the evaluation of ${\rm pfaff}(M_B)$ at $x_i=p_i+\epsilon v_i$ 
equals $\frac{1}{s!}(\alpha+\epsilon\beta)^{(s)}$. This expression expands to a linear combination of exterior monomials of the form $\epsilon^{s-t}(\beta^{(s-t)}\wedge \alpha^{(t)})$. All these monomials are divisible by $\epsilon^{s-1}$ since $\alpha^{(t)}=0$ for $t\geq 2$.

It remains to show that ${\rm pfaff}(M_B)$ is nonzero. Using the notation from 
Section~\ref{sec: identity}, we specialize  
$\begin{pmatrix}
X_1 & \! \cdots \! & X_n \\
x_1 & \! \cdots \! & x_n\\
\end{pmatrix}$
to
$\begin{pmatrix}
1 \! & \!\cdots \!     & \!1 \! & \! 1 \!& \! 0\\
x_1 \! & \!\cdots \! & \! x_{n-2} \!&\! 0 \! &\! 1\\
\end{pmatrix}$ and we define $Y_i$ and $P_i$ similarly. 
The specialization of the matrix $M$ 
is the matrix $\mathcal{A}(x,y,p)$ from Section~\ref{sec: identity}. Hence
\[{\rm pfaff}(M_B) \, \, = \,\, F_B(x,y,p) \,\, = \,\,\,
\frac{\Psi_B(x,p)\Psi_B(y,p)}{\prod_{i<j\in B}p_{ij}} , \]
where the second equality follows from Theorem~\ref{thm: identity}. Since $y$ is generic, it suffices to show that the function $x \mapsto \Psi_B(x,p)$ is not identically zero.
  Setting $x_i=p_i^{s}$ we recognize this as a Vandermonde determinant.
  It is nonzero as the $p_i$ are distinct. \end{proof}

Via Gale duality (see~\cite{Eisenbud-Popescu} for a detailed treatment) there is a correspondence between $n$-tuples of general points in $\mathbb{P}^{n-3}$ up to the action of ${\rm PGL}_{n-2}$, and $n$-tuples of general points in $\mathbb{P}^1$ up to the action of ${\rm PGL}_{2}$. We represent the ${\rm PGL}_{b+1}$ orbit of an $n$-tuple of points in $\mathbb{P}^b$ as a $(b+1)\times n$ matrix whose columns are the homogeneous coordinates of the points. In this language, Gale duality maps the orbit of the $(n-2)\times n$ matrix with columns $Q_1,\ldots, Q_n$ to the orbit of its kernel, represented by a $2\times n$ matrix. Via this correspondence we can also think of $n$ points in $\mathbb{P}^{n-3}$ as specifying a point $p$ in the Grassmannian ${\rm Gr}(2,n)$, up to the action of the $n$-dimensional diagonal torus.

Fix a point $p = (p_{ij})$ in ${\rm Gr}(2,n)$ that is Gale dual to the given $n$-tuple $Q_1,\ldots, Q_n$
in $\mathbb{P}^{n-3}$, 
and let  $y = (y_{ij})$ be a general point of ${\rm Gr}(2,n)$. As in Section~\ref{sec: identity} 
let $\mathcal{A}(x,y,p)$ denote
the skew-symmetric $n\times n$ matrix whose off-diagonal entries are
\[
\mathcal{A}(x,y,p)_{ij} \,\,\, = \,\,\, \frac{x_{ij}y_{ij}}{p_{ij}}
\,\,\, = \,\,\,\, \begin{vmatrix}
X_i & X_j\\
x_i & x_j\\
\end{vmatrix} \cdot
\begin{vmatrix}
Y_i & Y_j\\
y_i & y_j \\
\end{vmatrix} \cdot
\begin{vmatrix}
P_i & P_j\\
p_i & p_j\\
\end{vmatrix}^{-1}.
\]
For any even subset $\,B\subset [n]\,$ let $\,F_B(x,y,p)\in k[x_1,\ldots, x_n,X_1,\ldots, X_n]\,$ be the 
pfaffian of the submatrix of $\mathcal{A}(x,y,p)$ with rows and columns indexed by $B$.  

\begin{theorem} \label{thm: pfaffgens} The Cox ring of $X_n(Q)$ is isomorphic to the subalgebra of $k[x_{ij}][T^{\pm}]$ generated by the 
elements $T^{(1-\frac{1}{2}|B|)}F_B(x,y,p)$ as $B$ runs over the even subsets of~$[n]$. 
\end{theorem}

\begin{proof} 
Let  $m=t_0t_1^{-1}\cdots t_{n-1}^{-1}$. We consider the subalgebra of the (Laurent) polynomial ring
$k[z_1,\dots, z_{n-2},t_0^{\pm},t_1^{\pm},\dots, t_{n-1}^{\pm}]$ generated by the $2\times 2$ minors of the matrix 
\[ C(z,t) \,\, = \,\, \begin{pmatrix}
t_1 & t_2 & \cdots & t_{n-2} & t_{n-1} & 0 \, \\
mt_1z_1 & mt_2z_2 & \cdots & mt_{n-1}z_{n-1} & 0 & 1\, \\
\end{pmatrix}.
\] 
This algebra is isomorphic to the homogeneous coordinate ring of the Grassmannian ${\rm Gr}(2,n)$ since it is generated by $2\times 2$ minors and has the correct dimension $2n-3$. 
In particular, we can assume that the given point $p \in {\rm Gr}(2,n)$
 is specified by a matrix $C(p,t^{(p)})$ where $p=(p_1,\dots, p_{n-3})$ and $t^{(p)}=(t_0^*,\dots, t_{n-1}^*)$ have entries in $k$. 
 It represents the standard coordinate points $Q_1,\dots, Q_{n-1}$ and $Q_n=[p_1:\dots:p_{n-2}]$. 

Let $N$ denote the skew-symmetric $n \times n$-matrix whose off-diagonal entries are
 $$ N_{ij} \,\, = \,\, \frac{C_{ij}(x,t) C_{ij}(y,t^{(y)})}{C_{ij}(p,t^{(p)})} ,$$
 where $t^{(y)}$ is a vector of new variables.
It follows from our reparametrization of ${\rm Gr}(2,n)$
that the algebra generated by the pfaffians $T^{(1-\frac{1}{2}|B|)}{\rm pfaff}(N_B)$
is isomorphic to the algebra in Theorem \ref{thm: pfaffgens},
which is generated by  the elements $T^{(1-\frac{1}{2}|B|)}F_B(x,y,z)$.

Furthermore, the subpfaffians of the matrix $N$ agree, up to multiplication by a nonzero constant, with the even subpfaffians of the skewsymmetric matrix 
\[ M = 
\begin{tiny}
\begin{pmatrix}
0 \! & \!\! \frac{t^{d_{12}}(x_2-x_1)(y_2-y_1)}{(a_2-a_1)} \! & \frac{t^{d_{13}}(x_3-x_1)(y_3-y_1)}{(a_3-a_1)} & \! \cdots \! & \frac{t^{d_{1,n-2}}(x_{n-2}-x_1)(y_{n-2}-y_1)}{(a_{n-2}-a_1)} & -\frac{t^{d_{1,n-1}}x_1y_1}{a_1} & t_1\\
\vdots \! & 0 & \frac{t^{d_{23}}(x_3-x_2)(y_3-y_2)}{(a_3-a_2)} & \! \cdots \! & \frac{t^{d_{2,n-2}}(x_{n-2}-x_2)(y_{n-2}-y_2)}{(a_{n-2}-a_{2})} & -\frac{t^{d_{2,n-1}}x_2y_2}{a_2} & t_2\\
\vdots \! & \vdots & 0 & \! \cdots \! & \frac{t^{d_{3,n-2}}(x_{n-2}-x_3)(y_{n-2}-y_3)}{(a_{n-2}-a_3)} & -\frac{t^{d_{3,n-1}}x_3y_3}{a_3} &t_3\\
\vdots \! & \vdots & \vdots & \! \cdots \! & \vdots & \vdots & \vdots\\
-t_1 \! & -t_2 & -t_3 &\! \cdots \! & -t_{n-2} & -t_{n-1} & 0\\
\end{pmatrix},
\end{tiny}\] 
where we abbreviate $t^{d_{ij}}=mt_it_j$ for $i<j<n$. 
In particular, the $2^{n-1}$ subpfaffians of $M$ 
and the $2^{n-1}$ subpfaffians of $N$ define isomorphic $\mathbb{Z}^{n+1}$-graded $k$-algebras.

By Lemma~\ref{lem: hypersurfaces}, the rational function $t_n^{(1-\frac{1}{2}|B|)}{\rm pfaff}_B(M)$ is a nonzero element of the graded component $\,\Gamma(D(B))t^{D(B)}\,$
of $\, {\rm Cox}(X_n(Q))$. Here
$D(B)$ denotes the class of the $(-1)$-divisor determined by the subset $B$ as in Lemma~\ref{lem: bijection}. Since these sections generate the Cox ring of $X_n(Q)$, 
by Castravet-Tevelev \cite{CT}, the result follows. \end{proof}

\section{Even determinantal generators and phylogenetic trees}\label{sec:trees}

In this section we express the Cox ring of $X_n(Q)$ as the subalgebra generated by the determinants $\Psi_B(x,p)$ from Section~\ref{sec: identity}. These determinantal generators are a sagbi basis, and 
this yields a degeneration to certain algebras associated to trivalent (phylogenetic) trees. Let 
$p = (p_{ij})$ be Gale dual to the $n$-tuple of points $Q_1,\dots, Q_n$.

\begin{lemma} The Cox ring of $X_n(Q)$ is isomorphic to the subalgebra of $k[x_{ij}][T^{\pm}]$ generated by the expressions $\,T^{(1-\frac{1}{2}|B|)}\Psi_B(x,p)\,$ 
as $B$ ranges over even subsets of $[n]$.
\end{lemma}

\begin{proof} Let $(y_{ij})$ be a generic point in ${\rm Gr}(2,n)$. By Theorem~\ref{thm: identity} we have the identity
\[T^{(1-\frac{1}{2}|B|)}\cdot F_B(x,y,p) \,\,\, = \,\,\,
T^{(1-\frac{1}{2}|B|)}\cdot \frac{\Psi_B(x,p)\Psi_B(y,p)}{\prod_{i<j\in B}p_{ij}}.\]
Since $\frac{\Psi_B(y,p)}{\prod_{i<j\in B}p_{ij}}$ is a nonzero scalar in $k$, the algebra generated by the above determinants is isomorphic to the algebra generated by the pfaffians $T^{(1-\frac{1}{2}|B|)}F_B(x,y,p)$. The latter algebra is isomorphic to ${\rm Cox}(X_n(Q))$,
as was shown in Theorem~\ref{thm: pfaffgens}.
\end{proof}

\begin{remark} Our determinants $\Psi_B(x,p)$ can be thought of as the even counterparts to the 
odd-sized determinantal generators of Castravet-Tevelev in \cite[Theorem 1.1]{CT}.
\end{remark}

We now let the point $p = (p_{ij})$ range over a Zariski-open subset $U$ of ${\rm Gr}(2,n)$,
 and we consider the family of algebras over $U$ defined by the above even determinants.
   The fiber at $p \in U $ is isomorphic to the Cox ring of $X_n(Q)$,
    where $Q$ and $p$ are related by Gale duality. In particular, the isomorphism type of the algebra is constant on the orbits of the $n$-dimensional torus on ${\rm Gr}(2,n)$. 
    An important role in the degenerations we shall construct from this family
    is played by the following
{\em  bi-Pl\"ucker expansions}.

\begin{lemma}\label{lem: biplucker} Let 
$B=\{ b_1<\dots <b_{2s}\} \subset [n]$. For any sequence $i_1,\ldots, i_{s}$ 
of distinct elements of $B$, the following identify holds. Here
the sum is over the $s!$ permutations 
$\sigma$ of $B$ that satisfy $\sigma(b_m)=i_m$ for $1\leq m\leq s$,
and we abbreviate $j_m:=\sigma(b_{s+m})$:
\begin{equation}
\label{expansion}
\Psi_B(x,p)\,\, = \,\,
\sum_{\sigma} {\rm sgn}(\sigma) \prod_{r=1}^s \bigl(\,x_{i_rj_r}\prod_{m\neq r} p_{i_rj_m}\bigr)
\end{equation}
\end{lemma}

\begin{proof} To simplify notation we assume $B=[2s]$ and $i_1,\dots, i_s=1,\dots, s$. We prove the 
statement by suitably specializing the identity  (\ref{OkadaTwo}) from Theorem~\ref{thm: identity}.
Let $Y_m=1$ for $1\leq m\leq s$ and $Y_m=0$ otherwise and let $y_m=1$ for $1\leq m\leq 2s$. 
Then we have 
\[ \Psi(y,p)\,\,\,= \,\,\, \bigl(\! \prod_{i<j \leq s}\! p_{ij}\,\bigr) \cdot \bigl(\prod_{s< i<j} \! p_{ij} \,\bigr).\]
The specialized $2 {\times} 2$-minors are
$y_{ij}=1$ when $1\leq i\leq s$ and $s+1\leq j\leq 2s$ and $y_{ij}=0$ otherwise. In particular, if $\mu$ is any matching of $B$ then $\prod_{(i,j)\in \mu}y_{ij}=0$ unless every edge of $\mu$ 
connects an index $i \leq s$ with an index $j \geq s+1$.

 Clearing denominators and expanding the pfaffian 
in the identity  (\ref{OkadaTwo}), we obtain
$$
\Psi(x,p)\Psi(y,p) \,\,  = \,\,
\sum_{\sigma} {\rm sgn}(\sigma) \prod_{r=1}^s \bigl(x_{i_rj_r}\prod_{m\neq r} p_{i_rj_m}\bigr)\bigl(\prod_{i<j\leq s}p_{ij}\bigr)\bigl(\prod_{s<i<j}p_{ij} \bigr).
$$
The product of the last two terms in parenthesis equals $\Psi(y,p)$, and
this can be canceled on both sides of the equation. This yields the desired identity (\ref{expansion}).
\end{proof}

Our aim now is to select one of the terms in (\ref{expansion}) as
the leading term of $\Psi_B(x,p)$. This is done by  applying
to $p$ a valuation of the field $k({\rm Gr}(2,n))$ which is trivial on $k$.
 According to Speyer-Sturmfels \cite{SpSt}, all such valuations are
classified by the points on the {\em tropical Grassmannian} $\mathbb{TG}(2,n)$.
One way to construct valuations is to fix a field homomorphism
$\omega :k({\rm Gr}(2,n)) \rightarrow k(t)$
and then compose it with the usual valuation
${\rm val}$ on rational functions.
The corresponding point on the tropical Grassmannian 
$\mathbb{TG}(2,n)$ has tropical Pl\"ucker coordinates $\,w_{ij} = {\rm val}(\omega(p_{ij}) )$.
After subtracting a large constant, and after replacing each
$w_{ij}$ by its negative $-w_{ij}$, these tropical Pl\"ucker coordinates
are precisely the distances in a {\em tree metric} on $[n]$. This was shown in
\cite[\S 4]{SpSt}.
In the following statement, we assume that the reader is familiar with
the usage of {\em phylogenetic trees} as in \cite{BW} and
tree metrics as in \cite[\S 7]{SX}.

\begin{theorem}
\label{evensagbi}
For any trivalent phylogenetic tree $\mathcal{T}$ with leaves labeled by $[n]$
there exists a point $w \in \mathbb{TG}(2,n)$ such that
the leading form of $ \Psi_B(x,p)$ for the
weights $w_{ij}$  equals
$x_{i_1 j_1}  x_{i_2 j_2} \cdots x_{i_s j_s}\,$ where 
$\, \{i_1,j_1\} \cup \{i_2,j_2\}  \cup  \cdots \cup  \{i_s,j_s\}\,$
is the unique partition  of $B$ into disjoint paths on the tree $\mathcal{T}$.
The algebra generators $ T ^{1-\frac{1}{2}|B|}\Psi_B (x,p)$ form a sagbi basis
for ${\rm Cox}(X_n(Q))$
with respect to these weights
and weight zero on $T$.
\end{theorem}

\begin{proof}
Fix any tree metric $(-w_{ij})$ compatible with the
phylogenetic tree $\mathcal{T}$, and consider an even subset $B\subset [n]$ 
of size $2s$. There exists a unique matching $\mu=\{(i_1,b_1),\dots, (i_s, b_s)\}$
of the taxa $B$ whose connecting paths on $\mathcal{T}$
are pairwise disjoint. Note that this is the matching on $B$ whose paths have the shortest total length. 
It is referred to in \cite[\S 3.1]{BW}
as the {\em network of paths} with {\em sockets} in $B$.

Expanding $\Psi_B$ along the set $i_1,\dots, i_s$ we obtain, by Lemma~\ref{lem: biplucker}, 
the equality
\[\Psi_B(x,p)\,\,= \,\,
\sum_{\sigma} {\rm sgn}(\sigma) \prod_{r=1}^s \left(x_{i_rb_r}\prod_{m\neq r} p_{i_rb_m}\right)\]
where the sum runs over all matchings $\{(i_1,b_1),\dots, (i_s,b_s)\}$ of $B$.
The coefficient of the Pl\"ucker monomial $x_{i_1b_1} x_{i_2 b_2} \cdots x_{i_sb_s}$ in
this expansion equals
\[\prod_{r=1}^s\prod_{m\neq r}^s p_{i_rj_m}\,=\,\frac{\prod_{r=1}^s\prod_{m=1}^sp_{i_rb_m}}{\prod_{r=1}^sp_{i_rb_r}}.\]
The weight of this scalar is the weight of the numerator (which independent of the matching) 
plus the negated weight of the denominator. But  $-\sum_{r=1}^s w_{i_r b_r}$ is the total 
length of the matching, which is minimized at the 
special matching $\mu$ above.

To show that the generators $ T ^{1-\frac{1}{2}|B|}\Psi_B (x,p)$ form
a sagbi basis for the weights $w$, we degenerate them further
using the diagonal monomial order on the unknowns 
$x_1,\ldots,x_n, X_1,\ldots,X_n$. Namely, we replace the Pl\"ucker coordinate
$x_{ij}$ by $x_i X_j - x_j X_i$ and we declare $x_i X_j$ to be
the leading term. Now, the leading forms are all monomials.
It can be checked that the toric algebra generated by these
initial monomials is precisely the binary Jukes-Cantor model
on the tree $\mathcal{T}$ as studied in \cite{BW}. By \cite[\S 7]{SX}, this toric algebra
has the same multigraded Hilbert function as the Cox ring.
This shows that the two-step degeneration described above is flat.
\end{proof}

\begin{remark} The previous result is an even counterpart of the tree degenerations of the odd Castravet-Tevelev generators obtained by Sturmfels and Xu in~\cite[Theorem 7.10]{SX}. These degenerations imply that the Cox ring of $X_n(Q)$ is a Koszul algebra. 
\end{remark}

\section{Spinor varieties and their Gr\"obner bases}

In this section we review what is known about the other main players
in Theorem \ref{thm:main}, namely, the spinor variety $S^+$
and its defining ideal $I_{\rm spin}$. In particular, we present
an explicit quadratic Gr\"obner basis of $I_{\rm spin}$ due to DeConcini-Procesi \cite{DP}.

Let $V$ be a $2n$-dimensional vector space. Fix a basis $f_1,\ldots, f_n,g_1,\ldots, g_n$ for $V$ and let $W=\langle f_1,\ldots ,f_n\rangle$. We endow $V$ with the quadratic form $Q(f,g)=\sum_{i=1}^n f_ig_i$. An $m$-dimensional subspace $U\subset V$ is called {\em isotropic} if the restriction of $Q(x,y)$ to $U$ is zero. The {\em orthogonal Grassmannians} are the varieties which parametrize isotropic subspaces of $V$. When $m=n$ there are two connected components of maximal isotropic subspaces. These are parametrized by the (isomorphic) spinor varieties $S^{+}$ and $S^{-}$. The varieties $S^{+}$ and $S^{-}$ are naturally embedded in the even and odd half-spin representations of $\mathfrak{so}_{2n}$ whose underlying vector spaces are $\bigwedge^{even}W$ and $\bigwedge^{odd}W$. They can be realized as the orbit of the highest weight vectors of these representations under the action of the simple
Lie group of type $D_n$. 

In particular, with the conventions of Section~\ref{sec:geometry}, the spinor variety $S^{+}$ 
is the orbit of
the highest weight vector $f_{\emptyset}$. The spinor ideal $I_{\rm spin}$ defining $S^+$ is the ideal generated by
all homogeneous polynomials in the kernel of the $k$-algebra homomorphism 
\[ k\left[\bigwedge^{even}(W)\right] \,:= \, k \bigl[ \,f_B \,: \, B  \subset [n], \,\# \sigma \,\,{\rm even} \bigr] \, \rightarrow \,
k \bigl[z_{ij} \,: \, 1 \leq i < j \leq n \bigr] \, =: \,k[z] \]
that takes the variable $f_B$ to the Pfaffian of the
skew-symmetric matrix $(z_{ij})$ whose rows and columns are indexed by $B$. The spinor variety has dimension $\frac{1}{2}n(n-1)$.

The following {\em quadratic Grassmann-Pl\"ucker relation} is an element of the ideal $I_{\rm spin}$:
\begin{equation}
\label{wick1}
\sum_{i=1}^t (-1)^i \, f_{\tau_i \sigma_1 \sigma_2 \cdots \sigma_r} \,
                                    f_{\tau_1 \cdots \tau_{i-1} \tau_{i+1} \cdots \tau_s}
\,+\,
\sum_{j=1}^s (-1)^j \, f_{\sigma_1 \cdots \sigma_{j-1} \sigma_{j+1} \cdots \sigma_r}
\, f_{\sigma_j \tau_1 \tau_2 \cdots \tau_s}
\end{equation}
Here $\sigma$ and $\tau$ are any subsets of $[n]$ whose
cardinalities $s$ and $t$ are odd. This quadratic identity among
pfaffians is known to physicists as 
{\em Wick's Theorem}. See  \cite[Proposition 7.3.4]{Mu} and
\cite[Lemma 6.1]{DP} for combinatorial and algebraic perspectives.

For example, if $\sigma = \{1,3,4,5,6\}$ and $\tau = \{2\}$ then
the above quadric equals
\begin{equation}
\label{wick2}
\underline{
f_{3456} \,f_{12}}
+ f_{1456} \, f_{23}
- f_{1356} \, f_{24}
+ f_{1346} \, f_{25}
- f_{1345} \, f_{26}
- f_{123456} \, f.
\end{equation}

We partially order the variables $f_\sigma$ in $k[\bigwedge^{even}W]$
by setting $\, f_\sigma \succeq f_\tau\,$ whenever
$ \# \sigma \geq \# \tau$ and $\sigma_i \leq \tau_i$ for $i = 1,\ldots,\# \tau$.
This poset is the restriction of {\em Young's lattice} to 
 the even subsets of $[n]$. We consider any linear extension of
Young's lattice and we fix the {\em reverse lexicographic} term order $\succeq$
on $k[\bigwedge^{even}W]$ which is induced by the chosen total ordering of the variables.
For instance, for $n = 6$,
$$ f_{123456} \succeq f_{1234} \succeq f_{1235} \succeq \cdots
\succeq f_{3456} \succeq f_{12} \succeq f_{13} \succeq \cdots \succeq f_{56}
\succeq f_{\emptyset}.
$$
The leading term for this reverse lexicographic order is underlined in (\ref{wick2}).

\begin{thm} \label{Procesi} {\rm \cite[\S 6]{DP}}
The initial ideal of $I_{\rm spin}$ with  respect to $\succeq$
is generated by the monomials
$f_\sigma \cdot f_\tau$ corresponding to
incomparable pairs in Young's lattice.
\end{thm}

To prove Theorem \ref{Procesi}, an explicit minimal 
Gr\"obner basis for $I_{\rm spin}$ is derived 
from the Grassmann-Pl\"ucker relations
(\ref{wick1}). That minimal Gr\"obner basis is not
reduced. It consists of the straightening relations in \cite[Lemma 6.2]{DP}.
For example, the quadric (\ref{wick2}) is a straightening relation,
so it is in the  minimal Gr\"obner basis. But it is not in the reduced 
Gr\"obner basis since the second term is also  in
the monomial ideal ${\rm in}(I_{\rm spin})$.
The spinor ideal $I_{\rm spin}$ is homogeneous with respect to the
$\mathbb{Z}^{n+1}$-grading 
$$ {\rm deg}(f_\sigma) \, = \, e_0 + \sum_{j \in \sigma} e_j .$$

\begin{example} \rm We here present the reduced Gr\"obner basis 
promised by Theorem \ref{Procesi} for $n = 6$.
It consists of $66 = 15 + 30 + 15 + 6$ homogeneous polynomial.
They lie in $61$ different degrees, but up to $S_6$-symmetry
there are only four classes:
\begin{small}
$$
\begin{matrix}
 \underline{f_{14} f_{23}} - f_{13} f_{24} + f_{12} f_{34} - f_{1234} f & & 
\mbox{in degree} \,\, (2,0,0,1,1,1,1) \\
 \underline{f_{1345} f_{12}} - f_{1245} f_{13} + f_{1235} f_{14} - f_{1234} f_{15} & &
\mbox{in degree} \,\, (2,0,1,1,1,1,2) \\
 \underline{f_{1236} f_{1245}} - f_{1235} f_{1246} + f_{1234} f_{1256} - f_{123456} f_{12} \!\!\!\!\!
& & \mbox{in degree} \,\, (2,1,1,1,1,2,2) 
\end{matrix}
$$
\end{small}
The central degree $(2,1,1,1,1,1,1)$ has six reduced Gr\"obner basis elements:
\begin{small}
\begin{eqnarray*} &
\!\! \underline{f_{2345} f_{16}} - f_{1345} f_{26} + f_{1245} f_{36} - f_{1235} f_{46} +
f_{1234} f_{56} - f_{123456} f, \qquad \qquad \quad \\ &
\!\! \underline{f_{2346} f_{15}} - f_{1346} f_{25} + f_{1246} f_{35} -f_{1236} f_{45} -
f_{1234} f_{56} + f_{123456} f, \qquad \qquad \quad \\ & 
\!\! \underline{f_{2356} f_{14}} - f_{1356} f_{24} + f_{1256} f_{34} + f_{1236} f_{45}-
f_{1235} f_{46} - f_{123456} f, \qquad \qquad \quad \\ &
\underline{f_{2456} f_{13}} {-} f_{1356} f_{24} {+} f_{1346} f_{25} {-} f_{1345} f_{26} {+}
f_{1236} f_{45} {-} f_{1235} f_{46} {+} f_{1234} f_{56} {-} f_{123456} f, \\ &
\underline{f_{3456} f_{12}}  {-} f_{1256} f_{34} {+} f_{1246} f_{35} {-} f_{1245} f_{36} {-}
f_{1236} f_{45} {+} f_{1235} f_{46} {-} f_{1234} f_{56} {+} f_{123456} f, \\ &
\! \underline{f_{1456} f_{23}} - f_{1356} f_{24} + f_{1346} f_{25} - f_{1345} f_{26} + 
f_{1256} f_{34} - f_{1246} f_{35} + \qquad \qquad \,\,\,\, \\ &
\qquad f_{1245} f_{36} + f_{1236} f_{45} - 
f_{1235} f_{46} + f_{1234} f_{56} \,-\, 2 \, f_{123456} \, f.
\end{eqnarray*}
\end{small}
The take-home message is that the ideal $I_{\rm spin}$ 
of the spinor variety $S^+$ is a well-understood object. It
comes equipped with an explicit quadratic Gr\"obner basis
that can be generated by combinatorial methods, even for 
considerably larger values of~$n$.
\end{example}

\section{The Cox ring inside the spinor variety}\label{sec: final}
In this section we prove Theorem \ref{thm:main}. We fix general points $Q_1,\dots, Q_n\in \mathbb{P}^{n-3}$, along with their Gale dual $p = (p_{ij})\in {\rm Gr}(2,n)$, and this 
specifies the open subset 
\[\mathcal{G}(p)\,\,:=\,\,
\bigl\{c \,\in {\rm Gr}(2,n): \Psi_{B}(c,p)\neq 0\text{ for all even $B\subset [n]$}\bigr\}.\]
We further fix one auxiliary point $y\in \mathcal{G}(p)$. By Theorem~\ref{thm: pfaffgens}, 
the homomorphism 
\begin{equation}
\label{maptocox}
k\biggl[\,\bigwedge^{even}W \biggr] \rightarrow {\rm Cox}(X_n(Q)) , 
\end{equation}
mapping $\,f_{B} \mapsto T^{1-\frac{1}{2}|B|}F_B(x,y,p)\,$ for even subsets 
$B\subset [n]$, is surjective. 

We define the Cox ideal $I_X$ of $X_n(Q)$ to be the kernel of this homomorphism. While the isomorphism type of ${\rm Cox}(X_n(Q))$ depends only on the points $Q$, the ideal $I_X$ depends on $Q$ {\bf and} on the choice of the auxiliary parameter $y$. Our first result is a higher-dimensional analogue of an embedding for universal torsors on del Pezzo surfaces in~\cite{Derenthal-Advances, SS1}.

\begin{prop}
\label{prop:epi} The ring epimorphism (\ref{maptocox}) determines an embedding of the
spectrum of the Cox ring of $X_{n}(Q)$ into the spinor variety $\,S^{+}\,$ inside 
$\,\bigwedge^{even}W \,\simeq \, k^{2^{n-1}}$.
\end{prop}

\begin{proof} Since the polynomials $F_B(x,y,p)$ 
are the subpfaffians of a skew-symmetric matrix, they satisfy the
Grassmann-Pl\"ucker relations (\ref{wick1}).
Moreover, the scaling of coordinates $f_B\rightarrow T^{1-\frac{1}{2}|B|}f_B$ 
multiplies each quadric (\ref{wick1}) by a power of $T$. \end{proof}

Proposition \ref{prop:epi} establishes the inclusion $I_X\supseteq I_{\rm spin}$, 
and hence the first part of Theorem \ref{thm:main}.
The explicit quadratic Gr\"obner basis of Theorem \ref{Procesi}
furnishes many relations that hold in the Cox ring.
Our next goal is to derive the much stronger relationship
between $I_X$ and $I_{\rm spin}$ expressed
in the second part of Theorem \ref{thm:main}. 
The idea is that the Cox ideal $I_X$ should be determined
by the spinor ideal $I_{\rm spin}$
     if we allow for additional parameters which account for the moduli of the varieties $X_n(Q)$. To describe this more precisely we introduce some notation. For $c\in \mathcal{G}(p)$ let $a(c)$ be the point of the diagonal torus in $\bigwedge^{even}W$ with
      $a(c)_B=\Psi_B(c,p)\cdot\Psi_B(y,p)^{-1}$, and let $\star$ be the action of this torus by componentwise multiplication.
 
\begin{prop} We have the following inclusion of ideals in $k \bigl[ \bigwedge^{even} W \bigr]$:
\begin{equation}
\label{inclusion2}
I_X \,\, \supseteq \,\, \sum_{c\in \mathcal{G}(P)} a(c)\star I_{{\rm spin}} .
\end{equation}
\end{prop}

\begin{proof} For any $c\in \mathcal{G}(p)$ and $B\subset [n]$ even, we have
\[a(c)_BF_B(x,y,p) \,= \, \frac{\Psi_B(c,p)}{\Psi_B(y,p)}\frac{\Psi_B(x,p)\Psi_B(y,p)}{\prod_{i<j\in B}p_{ij}}
\,= \, \frac{\Psi_B(x,p)\Psi_B(c,p)}{\prod_{i<j\in B}p_{ij}} \,= \,F_B(x,c,p) ,\]
where the first and last equality follow from Theorem~\ref{thm: identity}.
Since the $F_B(x,c,p)$ are the subpfaffians of a skew-symmetric matrix,
 it follows that $a(c)\ast I_{\rm spin}\subseteq I_X$. Since $c\in \mathcal{G}(p)$ was arbitrary,
 we conclude that $\,\sum_{c\in \mathcal{G}(P)} a(c)\star I_{{\rm spin}}$
 is contained in $I_X$. \end{proof}

We expect the above inclusion to be, in general, an equality. We now show that this is the case in several special cases. By~\cite[\S 7]{SX} the ring ${\rm Cox}(X_n(Q))$ is a Koszul algebra so $I_X$ is generated by quadrics and thus proving the equality reduces to showing that both ideals have the same number of linearly independent quadrics.

The ideals $I_X$ and $\sum a(c)\star I_{{\rm spin}}$ are homogeneous with respect to the following 
(isomorphic) multigradings of 
$k\left[\bigwedge^{even}W\right]$, which refine the grading by total degree:
\[
\begin{array}{ll}
\deg(f_B) \,= \, g_0+\sum_{b\in B} g_b & \text{ for a basis $g_0,\dots, g_n$ of $\mathbb{Z}^{n+1}$, }\\
\deg(f_B) \,\in\, {\rm Pic}(X_n(Q)) & \text{ as in Lemma~\ref{lem: bijection}. }\\
\end{array}
\]
Hence, to verify the equality in (\ref{inclusion2}), it suffices to show that both 
ideals have the same number of linearly independent quadrics in each quadratic multidegree. 

\begin{lemma}\label{lem: sections} 
Up to the action of the Weyl group $D_n$,
there are precisely $\lfloor \frac{n}{2} \rfloor+1$ quadratic multidegrees
in $k[\wedge^{even} W]$.
 A system of distinct representatives is 
 given by the degrees $\,N_s=\deg(f_{\emptyset}f_{\{1,\dots, 2s\}})$ for $0\leq 2s\leq n$.
For $s>0$, the graded component of the Cox ring of $X_n$ in multidegree $N_s$ 
is a $k$-vector space of dimension $2^{s-1}$.
\end{lemma} 

\begin{proof}  Let $f_Af_B$ be a monomial in some quadratic multidegree. By transitivity of the action of the Weyl group  on $(-1)$-divisors we can assume that $A=\emptyset$. Moreover, a transposition $(ij)$ of two indices in $[n]$ is an element of the Weyl group. It corresponds to the action of a Cremona transformation of $\mathbb{P}^{n-3}$ centered at the points labeled by $[n] \backslash \{ij\}$. It follows that we can assume $B=\{1,\dots, 2s\}$ for some even $s$,
and the multidegrees $N_0, N_1, \ldots, N_{\lfloor \frac{n}{2} \rfloor+1}$
represent all orbits. The last statement is the content of~\cite[Corollary 7.4]{SX}.
It also shows that the $N_s$ lie in distinct $D_n$-orbits.
\end{proof}

\begin{theorem} \label{thm: generic} Suppose $n\leq 8$. For a generic $X_n(Q)$
there exists $c\in \mathcal{G}(p)$ such that
\[I_X=I_{\rm spin}+a(c)\star I_{{\rm spin}}.\] 
\end{theorem}
\begin{proof} It suffices to show that in all quadratic multidegrees the dimension of the quotient of $k[\bigwedge^{even}W]$ modulo the ideal $I_{\rm spin}+a(c)\star I_{{\rm spin}}$ is at most the one specified in Lemma~\ref{lem: sections}. Since this is an open condition, it suffices to verify this
claim for one choice of point $p = (p_{ij}) \in {\rm Gr}(2,n)$. We verify this by direct computation using the computer program Macaulay2 of Grayson and Stillman~\cite{M2}. The code and its output
are posted at our website \url{www.math.berkeley.edu/~velasco/StVe.html}.
\end{proof}


\end{document}